\documentclass[a4paper]{amsart}
\usepackage{amsmath}
\usepackage{amssymb}
\usepackage{amsthm}
\usepackage{color}
\usepackage{bm}
\usepackage{stmaryrd}
\usepackage{url}
\usepackage{cleveref}
\usepackage{graphicx}
\usepackage{a4wide}

\usepackage{cancel}

\usepackage{vem_macros}
\newtheorem{lemma}{Lemma}[section]
\newtheorem{theorem}{Theorem}[section]
\newtheorem{remark}{Remark}[section]

\usepackage{amssymb,color}
\usepackage[normalem]{ulem}
\definecolor{violet}{rgb}{0.580,0.,0.827}

\definecolor{purple}{RGB}{150, 0, 255}
\definecolor{orange}{RGB}{255, 150, 0}
\let\emptyset\varnothing

\newtheorem{definition}[theorem]{Definition}
\newtheorem{assump}[theorem]{Assumption}

\numberwithin{equation}{section}

\allowdisplaybreaks

\title[VEM for Quasilinear Problems]{A Virtual Element Method\\ for Quasilinear Elliptic Problems}
\author[A.~Cangiani]{Andrea Cangiani} \address{
A.~Cangiani, Department of  Mathematics,
University of Leicester,
University Road,
Leicester, LE1 7RH,
United Kingdom} 
\email{Andrea.Cangiani@le.ac.uk}
\author[P. Chatzipantelidis]{Panagiotis Chatzipantelidis}
\address{P. Chatzipantelidis, Department of Mathematics and Applied Mathematics, University of Crete, Heraklion, 71003, Crete, Greece}
\email{chatzipa@math.uoc.gr}
\author[G.~Diwan]{Ganesh Diwan}
\address{G. Diwan, Department of Medical Physics \& Biomedical Engineering,
University College London, Gower Street, London WC1E 6BT
}
\email{g.diwan@ucl.ac.uk}
\author[E.~H.~Georgoulis]{Emmanuil H.~Georgoulis} 
\address{
E.~H.~Georgoulis, Department of  Mathematics,
University of Leicester,
University Road,
Leicester, LE1 7RH,
United Kingdom, 
and Department of Mathematics, School of Applied Mathematical and Physical Sciences, National Technical University of Athens, Zografou 15780, Greece} \email{Emmanuil.Georgoulis@le.ac.uk}

\date{}

\begin{document}

\begin{abstract}
A Virtual Element Method (VEM) for the  quasilinear equation 
$-\text{div} (\diff(u)\text{grad} u)=f$
 using general polygonal and polyhedral meshes is presented and analysed. 
The nonlinear coefficient is evaluated with the piecewise polynomial projection of the virtual element ansatz. 
Well-posedness of the discrete problem and optimal order a priori error estimates in the $H^1$- and $L^2$-norm are proven. In addition, the convergence of fixed point iterations for the resulting nonlinear system is established. Numerical tests confirm the optimal convergence properties of the method on general meshes.
\end{abstract}

\maketitle

\section{Introduction}

In this work we present an arbitrary-order conforming Virtual Element Method (VEM)  for the numerical treatment of quasilinear diffusion problems. 
Both two and three dimensional problems are considered and the method is analysed under the same mesh regularity assumption used in the linear setting~\cite{BasicPaper,Confnonconf}, allowing for very general polygonal and polyhedral meshes.  

Virtual element methods for linear elliptic problems are now well-established, see eg.\cite{BasicPaper,ArbitraryRegularityVEM,EquivalentProjectors,General,NonconformingVEM,Confnonconf,Brenner_VEM} and~\cite{Sutton50lines} for a simple implementation. See also~\cite{VEMstability} for an extension to meshes with arbitrarily small edges and~\cite{VEMapost}  where the mesh generality is exploited within an adaptive algorithm driven by rigorous \emph{a posteriori} error estimates.
While the VEM framework has been concurrently extended to a number of different problems and applications,
 the literature on VEM for nonlinear problems is scarce, the same being true for other approaches to polygonal and polyhedral meshes also. 
The Cahn-Hilliard problem is considered in~\cite{CHVEM}, the stationary Navier-Stokes problem in~\cite{NSVEM}, and inelastic problems in~\cite{InelasticVEM}. However, the first two problems are semilinear, while for the (quasilinear) latter problem no analysis is provided. 
The related nodal Mimetic Finite Difference method is analysed in~\cite{AntoniettiBigoniVerani} for elliptic quasilinear problems whereby the nonlinear coefficient depends on the gradient of the solution, however only low-order discretisations are considered. We also mention the arbitrary order Hybrid High-Order method on polygonal meshes for the general class of Leray-Lions elliptic equations~\cite{DiPietroDroniou}, including the problems considered here. 
The HHO method belongs to the class of nonconforming/discontinuous discretisations and is, in fact, related to the Hybrid Mixed Mimetic approach and to the nonconforming VEM~\cite{UnifiedMFD,Bridging2016}. 
In~\cite{DiPietroDroniou}, the convergence of  HHO is proven under minimal regularity assumptions, but the rate of convergence of the method is not analysed.

The VEM presented here is based on the $C^0$-conforming virtual element spaces of~\cite{EquivalentProjectors} whereby the local $L^2$-projection of virtual element functions onto polynomials is available and the VEM proposed in~\cite{Confnonconf} for the discretisation of linear elliptic problems with non-constant coefficients. In particular, to obtain a practical (computable) formulation, the nonlinear diffusion coefficient is evaluated with the element-wise polynomial projection of the virtual element ansatz. 
This results in nonlinear inconsistency errors which have to be additionally controlled. 

We present an a priori analysis of the VEM which builds upon and extends the classical framework introduced by Douglas and Dupont~\cite{DD75} for standard conforming finite element methods.
The analysis relies on the assumption that the nonlinear diffusion coefficient is bounded and Lipschitz continuous and is based on a bootstrapping argument: 1. existence of  solutions for the numerical scheme is shown by a fixed point argument, 2. the $H^1$-norm error is bounded by optimal order terms plus the $L^2$-norm error, 3.  using a standard duality argument and assuming that the discretisation parameter is small enough, the $L^2$-norm error is bounded by optimal order terms plus potentially higher-order terms, 4. based on the existence result, $L^2$-convergence is shown by a compactness argument, and now $H^1$-convergence follows from step 2.
Within this approach, we also obtain optimal order a priori error estimates in the $H^1$- and $L^2$-norms, albeit under the (higher) regularity assumptions needed by the duality argument. 
To the best of our knowledge, this work provides the first optimal order error estimate for a conforming discretisation of quasilinear problems on general polygonal and polyhedral meshes.


To simplify the presentation, we consider homogeneous Dirichlet boundary value problems only. To this end, we introduce the model  quasilinear  elliptic problem
\begin{equation}\label{eq:pde}
  -\nabla\cdot(\diff(\u)\nabla \u)  = f(\vx) \ \text{ in }~\D, \quad \text{ with }\quad
  \u = 0\ \text{ on }~\dD,
\end{equation}
where  $\D \subset \Re^\spacedim$ is a convex polygonal or polyhedral domain for $\spacedim=2$ or $\spacedim=3$, respectively.
The diffusion coefficient is a twice differentiable function $\diff:\Re\rightarrow [\elipLower ,\elipUpper]$ such that  $0<\elipLower \le \elipUpper <+\infty$, and with bounded derivatives up to second order.  
Therefore $\diff$ is Lipschitz continuous, namely there exists a positive constant $L$ such that
\begin{equation}
	|\diff(t)-\diff(s)|  \leq L|t-s|,\qquad \text{for a.e } t,s\in \Re.
	\label{eq:diffLip}
\end{equation}
Writing \eqref{eq:pde} in variational form, we seek $\u \in H^1_0(\D)$ such that
\begin{equation}
	\A(\u;\u,\v):=(\diff(\u) \nabla \u, \nabla \v) = (f, \v), \quad \quad \forall \v \in H^1_0(\D),
	\label{eq:origVariationalForm}
\end{equation}
with $(\cdot,\cdot)$ denoting the standard $L^2$ inner-product. 
It is well known that for sufficiently smooth $f$, problem \eqref{eq:pde} possesses a unique solution $u$, see eg. \cite{DDS71}.

The remainder of this work is structured as follows. 
We introduce the virtual element method in~\sref{sec:vem}. The method is then analysed in~\sref{sec:analysis}, where the well-posedness and a priori analysis are presented. 
In~\sref{sec:iteration} we establish the convergence of fixed point iterations for the solution of the nonlinear system resulting from the VEM discretisation. 
We present a numerical test in~\sref{sec:num} and, finally, we provide some conclusions in~\sref{sec:conc}. 

We use standard notation for the relevant function spaces. For a Lipschitz domain $\omega \subset {\mathbb R}^d$, $d=2,3$, we denote by $|\omega|$ its $d$--dimensional Hausdorff measure.
Further, we denote by $H^s(\omega)$ the Hilbert space of index $s\ge 0$ of real--valued functions defined on $\omega$, endowed with the seminorm $|\cdot |_{s,\omega}$ and norm $\|\cdot\|_{s,\omega}$; further $(\cdot,\cdot)_\omega$ stands for the standard $L^2$-inner-product. The domain of definition will be omitted when this coincides with $\Omega$, eg. $|\cdot |_{s}:=|\cdot |_{s,\Omega}$ and so on. Finally, for $\ell\in {\mathbb N}\cup \{0\}$, we denote by ${\mathbb P}_{\ell}(\omega)$ the space of all polynomials of degree up to $\ell$.

\section{The Virtual Element Method}
\label{sec:vem}

We introduce the virtual element method for the discretisation of problem~\eqref{eq:origVariationalForm}, using general polygonal and polyhedral decompositions of $\Omega$ in two and three dimensions, respectively. 
We start by recalling the definition of the virtual element spaces from~\cite{EquivalentProjectors,Confnonconf}.

\subsection{The Discrete Spaces}\label{sec:vemSpaces}
The definition of the virtual element method relies on the availability of certain local projector operators based 
on accessing the degrees of freedom. The choice of degrees of freedom for the virtual element spaces is thus 
important.
\begin{definition}[Degrees of freedom]
  \label{def:dofs}
  Let $\omega\subset\Re^\spacedim$, $1\le \spacedim\le 3$, be a $\spacedim$-dimensional
  polytope, that is, a line segment, polygon, or polyhedron,
  respectively.  For any regular enough function $v$ on $\omega$, we define the following sets of \emph{degrees of freedom}:
  \begin{itemize}
  \item \emph{Nodal values}. For a vertex $\vertex$
    of $\omega$, $\NO{\omega}_\vertex (v):=v(\vertex)$ and
    $\NO{\omega}:=\{\NO{\omega}_\vertex: \vertex \text{ is a
      vertex}\}$;
  \item \emph{Polynomial moments}. For $l\ge
    0$,
    \begin{equation*}
      \MO{\omega}{\mindex}(v)=\frac{1}{\abs{\omega}} (v, \ma)_\omega
      \quad \text{ with}\quad m_{\mindex}
      :=
      \(\frac{ \vx - \vx_{\omega}}{h_{\omega}}\)^{\mindex} \text{ and}\quad \abs{\mindex}\le l,
    \end{equation*}
    where $\mindex$ is a multi-index with $\abs{\mindex} := \alpha_1
    +\cdots +\alpha_\spacedim$ and $\vx^{\mindex} := x_1^{\alpha_1} \dots
    x_\spacedim^{\alpha_\spacedim}$ in a local coordinate system, and $\vx_\omega$
    denoting the barycentre of $\omega$. Further,
    $\MO{\omega}{l}:=\{\MO{\omega}{\mindex}:\abs{\mindex}\le l\}$. The
    definition is extended to $l=-1$ by setting
    $\MO{\omega}{-1}:=\emptyset$.
  \end{itemize}
\end{definition}

Let $\{\Th\}_h$ be a sequence of decompositions of  $\Omega$ into non-overlapping and not self-intersecting polygonal/polyhedral elements such that the diameter of any $\E\in\Th$ is bounded by $h$. 

On $\Th$, we introduce element-wise projectors as follows.
We denote by ${\Po{\ell}}\equiv \Po{\ell,\E} : {L^2(\E)} \rightarrow \PE{\ell}$, $\ell\in\Na$, the standard $L^2(\E)$-orthogonal projection onto the polynomial space $\PE{\ell}$.
With slight abuse of notation, the symbol ${\Po{\ell}}$ will also be used to denote the global operator obtained from the piecewise projections.
%
Similarly, by $\vPo{\ell}\equiv\vPo{\ell,E} $, $\ell\in\Na$, we denote the orthogonal projection of $(L^2(E))^d$ onto the  space $\vPE{\ell}=(\PE{\ell})^\spacedim$, obtained by applying $\Po{\ell,\E}$ component-wise.
Further,  we consider the projection $\PN{\ell}\equiv \PN{\ell,E}:H^1(E)\to \PE{\ell}$, for $\ell\in\Na$,
associating any $v\in H^1(\E)$ with the element in $\PE{\ell}$ such that
\begin{equation}
	\label{eq:h1ProjDefn}
	(\nabla \PN{\ell} v, \nabla p)_{\E} = (\nabla v, \nabla p)_{\E}, \quad \forall p \in \PE{\ell},
\end{equation}
with, in order to uniquely determine $\PN{\ell}$, the additional condition:
\begin{align}\label{eq:h1ProjDefn-2}
\begin{cases}
	\displaystyle\intdE \big(v - \PN{\ell} v\big) \ds = 0  & \text{if } \ell = 1, \\[1em]
	\displaystyle\intE  \big(v - \PN{\ell} v\big) \dx = 0  & \text{if } \ell \geq 2.
\end{cases}
\end{align}

Let $k\ge 1$ be given, characterising the order of the method. We follow the construction of the corresponding $C^0$-conforming VEM space  presented in~\cite{EquivalentProjectors} to ensure that 
all of the above projectors, to be utilised in the definition of the method, are computable.

We first introduce the local spaces on each element $\E$ of $\Th$, for $\spacedim = 2$.
Let $B^2_{\k}(\dE)$  be the space defined on the boundary of $\E$ in the following way
\begin{equation*}
	B^2_{\k}(\dE) := \left\{ \v \in C^0(\dE) : \v|_\e \in \P{\e}{\k} \text{ for each edge } \e \text{ of } \dE \right\}.
\end{equation*}

We define the local virtual element space $\VhE$ by
\begin{align*}
	\VhE := \{ \vh \in H^1(\E) : \,&\vh|_{\dE} \in B^2_{\k}(\dE); \,\, \Delta \vh \in \PE{\k}  \\
		&\text{ and } (\vh - \PN{k}  \vh, p )_{\E} = 0, \,\forall p \in \ME{\k} \setminus \ME{\k-2} \}.
\end{align*}

In \cite{EquivalentProjectors} it is shown that the following degrees of freedom (DoF) uniquely determine the elements of $\VhE$:\begin{equation}\label{eq:dof2}
\text{DoF}(\VhE)
:=
\NO{\E}
\cup
\{\MO{\e}{\k-2}: \text{ for each edge } \e\in\dE\}
\cup
\MO{\E}{\k-2}.
\end{equation} 

The global conforming space $\Vh$ is  obtained from the local spaces $\VhE$ as
\begin{equation*}
	\Vh := \left\{ \vh \in H^1_0(\D) : \vh |_{\E} \in \VhE, \quad \forall \E \in \Th \right\},
\end{equation*}
with degrees of freedom given in agreement with the local degrees of freedom~\eqref{eq:dof2}.

The construction of the space for $\spacedim=3$ is similar, although now we define the boundary space to be
\begin{equation*}
	B^3_{\k}(\dE) := \left\{ \v \in C^0(\dE) : \v|_\f \in \Vhf \text{ for each face } \f \text{ of } \dE \right\},
\end{equation*}
where $\Vhf$ is the two-dimensional conforming virtual element space of the same degree $\k$ on the face $\f$. The local virtual element space is defined to be
\begin{align*}
	\VhE := \{ \v \in H^1(\E) : \,&\v|_{\dE} \in B^3_{\k}(\dE); \,\, \Delta \v \in \PE{\k};  \\
		&\text{ and } (\v - \PN{\k} \v, p )_{\E} = 0, \,\,\forall p \in \ME{\k} \setminus \ME{\k-2} \}.
\end{align*}
with degrees of freedom 
\begin{equation}\label{eq:dofs3d}
\text{DoF}(\VhE)
:=
\NO{\E}
\cup
\{\MO{\s}{\k-2} \text{ for each edge \emph{and} face } \s\in\dE\}
\cup
\MO{\E}{\k-2}.
\end{equation}
Finally, the global space and the set of global degrees of freedom for $\spacedim=3$ are constructed from these in the obvious way, completely analogously to the case for $\spacedim=2$.

The following are well established properties of the virtual element spaces introduced above~\cite{BasicPaper,EquivalentProjectors,Confnonconf}:
\begin{itemize}
\item For each $\E\in\Th$, we have  $\PE{\k}\subset\VhE $ as a subspace;
\item For each $\E\in\Th$ and $\v\in\VhE$, the $H^1$-projector $\PN{k,E}\v$ and $L^2$-projectors $\Po{k,\E}\v$ and $\vPo{k-1,E}\nabla\v$ are computable just by accessing the local DoFs of $\v$ given by~\eqref{eq:dof2} and~\eqref{eq:dofs3d} in the two and three dimensional case, respectively.
\item The global virtual element space $\Vh\subset H^1_0(\Omega)$ as a finite dimensional subspace.
\end{itemize}

\subsection{Virtual element method}
The virtual element method of order $k\ge 1$ for the discretisation of~\eqref{eq:pde} reads: find $\uh\in \Vh$ such that 
\begin{equation}\label{vem}
a_h(\uh;\uh,\vh)=(\Po{\k-1} f,\vh), \quad\forall \vh\in V_h, 
\end{equation}
where $a_h(\cdot;\cdot,\cdot)$ is \emph{any} bilinear form on $\Vh$ defined as the sum of elementwise contributions $a_h^\E(\cdot;\cdot,\cdot)$ satisfying the following assumption~\cite{BasicPaper}.
\begin{assump}\label{ass:bilinearForms}
	For every $\E\in\Th$, the form $\ahE(\cdot;\cdot, \cdot)$ is bilinear and symmetric in its second and third arguments and satisfies the following properties:
	\begin{itemize}
		\item \emph{Polynomial consistency}:
			For all $p \in \PE{\k}$ and $\vh \in \VhE$, 
			\begin{equation}\label{eq:consistency}
				\ahE(\z;p, \vh) = \intE \diff(\Po{}\z) \nabla p\cdot (\vPo{} \nabla \vh) \dx, 	\quad\forall \z\in L^2(E),		\end{equation}
				where $\Po{}=\Po{\k}$ and $\vPo{}=\vPo{\k-1}$.
                          		\item \emph{Stability}:
			There exist positive constants $\diffStabLower, \diffStabUpper$, independent of $h$ and the mesh element $\E$ such that, for all $\vh,\zh \in \VhE$,
			\begin{equation}\label{stability}
				\diffStabLower \aE(\zh;\vh, \vh) \leq \, 
				\ahE(\zh;\vh, \vh) \leq \diffStabUpper \aE(\zh;\vh, \vh), 
				\end{equation}
				with $\aE(\z;\v,\w)=(\diff(\z)  \nabla \v, \nabla \w)_{\E}$, for all $\z\in L^\infty(\Omega)$ and $\v,\w\in H^1(\Omega)$. 
	\end{itemize}
\end{assump}

\begin{remark}
The above defining conditions are essentially those introduced  in the linear setting~\cite{BasicPaper,ArbitraryRegularityVEM,EquivalentProjectors,General,Confnonconf} with, crucially, the nonlinear diffusion coefficient $\diff$ evaluated with the polynomial projection of the argument. We note also that the symmetry and stability assumptions imply the continuity  in $\Vh$ of the form $\ah(\z;\cdot, \cdot)$, for $\z\in\Vh$.
\end{remark}
\begin{remark}
The particular choice of local bilinear forms used in the numerical tests is given below in~\sref{sec:num}. We remark, however, that the following error analysis is valid whenever the assumption above is  satisfied.
\end{remark}

\section{Error Analysis}
\label{sec:analysis}

We recall that $k\ge 1$ is a fixed natural number representing the order of accuracy of the method~\eqref{vem}.

The convergence and a priori error analysis of the VEM relies on the availability of the following best approximation results. 
\subsection{Approximation Properties}
\label{subsec:approximation-properties}
We recall the optimal approximation properties of the VEM space $\Vh$ introduced above. These where established in a series of papers \cite{BasicPaper,EquivalentProjectors,VEMapost} under the following assumption on the regularity of the decomposition $\Th$.
\begin{assump}{(Mesh Regularity).}
	\label{ass:meshReg}
	We assume the existence of a constant $\meshReg > 0$ such that
	\begin{itemize}
		\item for every element $\E$ of $\Th$ and every edge/face $\e$ of $\E$, $h_{\e} \geq \meshReg h_{\E}$
		\item every element $\E$ of $\Th$ is star-shaped with respect to a ball of radius $\meshReg h_{\E}$
		\item for $\spacedim=3$, every face $\e \in \Edges$ is star-shaped with respect to a ball of radius $\meshReg h_{\e}$,
	\end{itemize}
	where $h_e$ is the diameter of the edge/face $\e$ of $\E$ and $h_E$ is the diameter of $\E$.
\end{assump}

The above star-shapedness assumption can be relaxed by including elements which are union of star-shaped domains~\cite{BasicPaper}. In particular, the following polynomial approximation result~\cite{BrennerScott} is extended to more general shaped elements in~\cite{ScottDupont} and the interpolation error bound below can be generalised by modifying the proof in~\cite{VEMapost}, see also~\cite{Oliver}.
\begin{theorem}[Approximation using polynomials]
	\label{thm:polynomialApproximation}
	Suppose that Assumption~\ref{ass:meshReg} is satisfied and let $s$ be a positive integer such that $1 \leq s \leq \k+1$. Then, for any $\w \in H^{s}(\E)$ there exists a polynomial $\w_\pi \in \PE{\k}$ such that
	\begin{equation*}
		\norm{\w - \w_\pi }_{0,\E} + h_{\E} \norm{\nabla(\w - \w_\pi )}_{0,\E} \leq Ch_{\E}^s \abs{\w}_{s,\E}.
	\end{equation*}
	Moreover, we have
	\begin{equation*}
	\norm{\nabla(\w - \w_\pi )}_{L^6(\E)} \leq C \abs{\w}_{W^{1,6}(\E)}.
	\end{equation*}
	In the above bounds, $C$ are positive constants depending only on $\k$ and on $\meshReg$.  
\end{theorem}

The approximation properties of the virtual element space are characterised by the following interpolation error bound, whose proof can be found in~\cite{VEMapost}.

\begin{theorem}[Approximation using virtual element functions]
	\label{thm:spaceApproximation}
	Suppose that Assumption~\ref{ass:meshReg} is satisfied and let $s$ be a positive integer such that $1 \leq s \leq \k+1$. Then, for any $\w \in H^s(\D)$, there exists an element $w_I \in \Vh$ such that
	\begin{equation*}
		\norm{\w - w_I} + h \norm{\nabla(\w - w_I)} \leq Ch^s \abs{\w}_{s}
	\end{equation*}
	where $C$ is a positive constant which depends only on $\k$ and $\meshReg$.
\end{theorem}

Let $\vep:L^2(\Omega)\times \Vh\to\mathbb{R}$ denote the bilinear form 
\begin{equation}\label{vep-def}
\vep(f,\vh)=({\Po{\k-1}}f-f,\vh),\quad\forall\vh\in\Vh.
\end{equation}
Then, using the fact that ${\Po{\k-1}}f$ is the $L^2$ projection on $\PE{\k-1}$, we can show the following lemma.
\begin{lemma}\label{lemma-vep}
For $f\in H^{s}(\Omega)$, $0\le s\le \k$, there exists a positive constant $C$, independent of $h$ and of $f$, such that  
\begin{equation}
|\vep(f,\vh)|\le Ch^{s+j}\|f\|_{s}\,\|\nabla^j\vh\|,\quad\forall\vh\in \Vh,\ j=0,1.
\end{equation}
\end{lemma}

\bigskip
\subsection{Existence} We first show the existence of a solution $\uh$ of \eqref{vem} using a fixed point argument. To this end, for  $M>0$, we let 
$
{\mathcal B}_M=\{\vh\in \Vh:\|\nabla\vh\|\le M\}.
$
\begin{theorem} \label{thm-existence}
Let $f \in L^2(\Omega)$ be given and assume that~\eqref{eq:diffLip} holds.
Choose $M>0$ such that  $\|f\|\le M{c_*}$, $c_*=\elipLower\diffStabLower$
where  $\diffStabLower$ is the lower bound constant 
in  \eqref{stability}. Then, there exists a solution $\uh\in{\mathcal B}_M\subset \Vh$ of~\eqref{vem}.
\end{theorem}

\begin{proof}
We devise a fixed point iteration for \eqref{vem}:
for a fixed $f\in L^2(\Omega)$, consider an iteration map $T_h:\Vh\to \Vh$
given by
\begin{equation}\label{elliptic-iteration}
a_h(\vh;T_h\vh,\wh)=(\Po{\k-1}f,\wh), \quad\forall\wh \in \Vh.
\end{equation}
It is easy to see that there exists $h_M>0$, such that for
$h<h_M$,
$T_h v_h$ is well defined, see for example \cite{Confnonconf}.
For $\vh\in \mathcal B_M$ and $\wh=T_h\vh$, in view of the stability assumption \eqref{stability} and \eqref{elliptic-iteration}, we have
\begin{equation}\begin{split}
c_\star{\|\nabla T_h\vh\|}^2\le 
\diffStabLower a(\vh;T_h\vh,\wh)\le a_h(\vh;T_h\vh,\wh)
=(\Po{\k-1}f,\wh)\le \|f\|\,\|\wh\|.
\end{split}\end{equation}
 Thus, choosing $M$ sufficiently large, so that $\|f\|\le Mc_\star$, we get
\begin{equation}\begin{split}
\label{elliptic-projection-error-1-v2} {\|\nabla
T_hv_h\|}\le c_*^{-1} \|f\| \le M.
\end{split}\end{equation}
Therefore, the operator $T_h$ maps the ball $v_h\in \mathcal B_M$ into itself.
By the Brouwer fixed point theorem, we know that
$T_h$ has a fixed point, which implies that \eqref{vem}
has a solution $\uh \in {\mathcal B}_M$.
\end{proof}

\subsection{Error bounds} In our a priori error analysis, we follow a similar-in-spirit approach to the classical work of Douglas and Dupont~\cite{DD75} where  standard conforming finite element methods were analysed in the same context.

We start with the following preliminary $H^1$--norm error bound.
\begin{theorem} \label{thm-H1-bound}
Let $u\in H^1_0(\Omega)$ be the solution of \eqref{eq:pde} and suppose that $u\in H^{s}(\Omega)\cap W^1_\infty (\Omega)$, $s\ge 2$, assuming that   $f\in H^{s-2}(\Omega)$ and $\diff(u)\in W^{s-1}_{\infty}(\Omega)$. 
Then, for $\uh\in\Vh$ solution of \eqref{vem} the following bound holds 
\begin{equation}\label{eq:thm-H1-bound}
\|\nabla(u-\uh)\|\le C(h^{ r-1}+\|u-\uh\|),
\end{equation}
with $r=\min\{s,k+1\}$ and $C$ a positive constant independent of $h$. 
\end{theorem}
\begin{proof} From Theorem \ref{thm:spaceApproximation}, there exists a function 
$\u_I\in\Vh$, such that $u-\u_I$ is bounded as desired.   Thus, to show \eqref{eq:thm-H1-bound} 
it  suffices to bound $\|\nabla(\uh-\u_I)\|$.
Let $\psi=\uh-\u_I$, then using the stability Assumption \ref{ass:bilinearForms} 
with $c_*=\elipLower\diffStabLower$, we have
\begin{align}
c_*\|\nabla(\uh-\u_I)\|^2&\le\ah(\uh;\uh-\u_I,\psi)\nonumber\\
&=\vep(f,\psi)+a(\u;\u,\psi)-\ah(\uh;\u_I,\psi)\nonumber\\
&=\vep(f,\psi)+((\diff(u)-\diff(\Po{}\uh))\nabla u,\nabla\psi)
 +\sum_{E\in\Th}\aE(\Po{}\uh; \u-\u_\pi,\psi)\nonumber\\
&\ +\left\{\sum_{E\in\Th}
\aE(\Po{}\uh;\u_\pi,\psi)
-\ahE(\uh;\u_\pi,\psi)\right\}
+\sum_{E\in\Th}\ahE(\uh;\u_\pi-\u_I,\psi)\nonumber\\
&=I_1+I_2+I_3+I_4+I_{5},\label{I-split}
\end{align}
where $\u_\pi$ is, on every element $\E\in\Th$,  the polynomial approximation of $u$ given by Theorem \ref{thm:polynomialApproximation}.
Next, we will bound the various terms $I_i$, $i=1,\dots,5$. We start with $I_1$. Using Lemma \ref{lemma-vep}, and the fact that $r\le s$, we have
\begin{equation}\label{I1-bound}
|I_1|\le C h^{r-1}\|f\|_{r-2}\|\nabla\psi\|.
\end{equation}
To bound $I_2$, in view of \eqref{eq:diffLip}, we get 
\begin{equation}\label{I2-bound}
|I_2|\le L\|\nabla u\|_{L_\infty}\|u-\Po{}\uh\|\,\|\nabla\psi\|.
\end{equation}
Also, using the fact that $\diff$ is bounded along with Theorem \ref{thm:polynomialApproximation}, we obtain
\begin{equation}\label{I3-bound}
|I_3|\le C\sum_E\|\nabla(\u-\u_\pi)\|_E\|\nabla \psi\|_E\le Ch^{ r-1}\|\u\|_{r}\|\nabla\psi\|.
\end{equation}
Using the fact that $\nabla \u_\pi\in \vPE{\k-1}$ and Assumption \ref{ass:bilinearForms}, we have
\begin{equation*}\begin{split}
I_4&=\sum_{E\in\Th}\int_{E}\diff(\Po{}\uh)\nabla \u_\pi \cdot(\bm I-\vPo{})\nabla\psi\\
&=\sum_{E\in\Th}\int_{E}\diff(\Po{}\uh)\nabla (\u_\pi-\u) \cdot(\bm I-\vPo{})\nabla\psi+\int_{E}\diff(\Po{}\uh)\nabla\u\cdot (\bm I-\vPo{})\nabla\psi\\
&=\sum_{E\in\Th}\int_{E}(\diff(\Po{}\uh)-\diff(u))\nabla (\u_\pi-\u)\cdot (\bm I-\vPo{})\nabla\psi+\int_{E}\diff(u)\nabla (\u_\pi-\u)\cdot(\bm I-\vPo{})\nabla\psi\\
&\quad+\sum_{E\in\Th}\int_{E}(\diff(\Po{}\uh)-\diff(u))\nabla\u\cdot (\bm I-\vPo{})\nabla\psi+\int_{E}(\bm I-\vPo{})(\diff(u)\nabla\u)\cdot\nabla\psi;
\end{split}\end{equation*}
thus, in view of the stability of $\vPo{}$, the fact that $\diff$ is Lipschitz continuous, $u\in W^{1}_{\infty}(\Omega)$, Theorem  \ref{thm:polynomialApproximation} and   the hypothesis $\diff(u)\in W^{r-1}_{\infty}(\Omega)$, we deduce
\begin{equation}\label{I4-bound}
\begin{split}
|I_4|&\le C\sum_{E\in\Th}(\|\nabla(\u-\u_\pi)\|_E+\|\Po{}\uh-\u\|_E)\|\nabla \psi\|_E
+\|(\bm I-\vPo{})(\diff(\u)\nabla \u)\|_{E}\|\nabla\psi\|_{E}\\
&\le C(h^{ r-1}\|\u\|_{ r}+\|\Po{}\uh-\u\|)\|\nabla\psi\|.
\end{split}\end{equation}

Finally, we easily get
\begin{equation}\label{I5-bound}
|I_5|\le C(\|\u-\u_\pi\|+\|\u-\u_I\|)\|\nabla\psi\|\le Ch^{ r}\|\u\|_{ r}\|\nabla\psi\|.
\end{equation}

Therefore,  combining the above estimates \eqref{I1-bound}--\eqref{I5-bound} with \eqref{I-split} we obtain
\begin{equation*}
c_\star\|\nabla(\uh-\u_I)\|\le C(h^{ r-1}+\|u-\Po{}\uh\|).
\end{equation*}

Then, in view of Theorem \ref{thm:polynomialApproximation} and the stability of $\Po{}$ in $L^2$--norm, we obtain the estimate
\begin{equation*}
\|\nabla(\uh-\u_I)\|\le C(h^{ r-1}+\|u-\uh\|).
\end{equation*}
\end{proof}

Next, we shall demonstrate the following preliminary $L^2$--norm, error bound.
\begin{theorem} \label{thm-L2-bound}
Let $u\in H^1_0(\Omega)$ be the solution of \eqref{eq:pde} and suppose that $u\in H^{s}(\Omega)\cap W^1_\infty (\Omega)$, $s\ge2$, assuming that $f\in H^{s-1}(\Omega)$ and $\diff(u)\in W^{s-1}_{\infty}(\Omega)$. Then,  for $h$ small enough and  $\uh\in\Vh$ solution of \eqref{vem} the following bound holds 
\begin{equation}\label{eq:thm-L2-bound}
\|u-\uh\|\le C(h^{r}+\|u-\uh\|^{3}),
\end{equation}
where $r=\min\{s,k+1\}$ and $C$ is a positive constant independent of $h$. 
\end{theorem}

\begin{proof} 
We use a duality argument. Consider the (linear) auxiliary problem: find $\phi\in H^1_0(\Omega)$ such that
$$
-\text{div}(\diff(u)\nabla \phi)+\diff_u(u)\nabla u\cdot\nabla\phi=u-\uh.
$$
Noting that this equates to $\diff(u)\Delta\phi=u-\uh$ and given $\Omega$ is convex, we have  $\phi\in H^2(\Omega)$ and
\begin{equation}\label{auxiliary-bound}
\|\phi\|_2\le C\|u-\uh\|.
\end{equation} 
In variational form, the above problem reads
\begin{equation}\label{aux-var}
(\diff(u)\nabla \phi,\nabla v)+(\diff_u(u)\nabla u\cdot\nabla\phi,v)=(u-\uh,v),\quad\forall v\in H^1_0(\Omega),
\end{equation}
Then choosing $v=u-\uh$ in \eqref{aux-var}
\begin{align}
\|u-\uh\|^2&=(\diff(u)\nabla\phi,\nabla(u-\uh))+(\diff_u(u)(u-\uh)\nabla u,\nabla \phi)\nonumber\\
&=(\diff(u)\nabla u,\nabla\phi)-(\diff(\uh)\nabla \uh,\nabla\phi)-((\diff(u)-\diff(\uh))\nabla \uh,\nabla\phi)\nonumber\\
&+(\diff_u(u)(u-\uh)\nabla u,\nabla \phi)\nonumber\\
&=(\diff(u)\nabla u,\nabla\phi)-(\diff(\uh)\nabla \uh,\nabla\phi)+((\diff(u)-\diff(\uh))\nabla (u-\uh),\nabla\phi)\nonumber\\
&-((\diff(u)-\diff(\uh))\nabla u-\diff_u(u)(u-\uh)\nabla u,\nabla\phi)\nonumber\\
&=\Big(a(u;u,\phi)-a(\uh;\uh,\phi)\Big)\nonumber\\
&+\Big(((\bar \diff_u(u-\uh)\nabla (u-\uh),\nabla\phi)-((\bar \diff_{uu}(u-\uh)^2\nabla u,\nabla\phi)\Big)=:I+II,\label{error-split}
\end{align}
with $\bar \diff_u, \bar \diff_{uu}$ such that
\begin{align}
\diff(u)-\diff(\uh)&=(u-\uh)\int_0^1\diff_u(u-t(u-\uh))\,dt=\bar \diff_u(u-\uh)\label{ku}\\
\diff(u)-\diff(\uh)-\diff_u(u)(u-\uh)&=(u-\uh)^2\int_0^1\diff_{uu}(u-t(u-\uh))\,dt\notag\\
&=\bar \diff_{uu}(u-\uh)^2\label{kuu}.
\end{align}
In the sequel we will show Lemma \ref{thm:consistency:term I}, which in view of \eqref{auxiliary-bound}, gives
\begin{equation}
\label{eq:consistency:term I}
|I| \le  C 
( h\|\nabla(\u-\uh)\|
 +\|\u-\uh\|^{1/2}\|\nabla(\u-\uh)\|^{3/2}
+h^{r}\|u\|_{r}+h^{r}\|f\|_{r-1})\|\uh-u\|.
\end{equation}
For $II$ in~\eqref{error-split}, using the H\"older inequality 
\begin{equation}\label{Holder-ineq}
\|vw\|\le \|v\|_{L_3}\|w\|_{L_6},
\end{equation}
and the fact that $\bar \diff_u, \bar\diff_{uu}$ are bounded uniformly on $\mathbb{R}$, we  get
\begin{align*}
|II|&\le C\|\nabla(u-\uh)\|\,\|(u-\uh)\nabla\phi\|+C\|(u-\uh)\nabla u\|\,\|(u-\uh)\nabla \phi\|\\
&\le C\|\nabla (u-\uh)\|\,\|u-\uh\|_{L_3}\|\nabla\phi\|_{L_6}
+C\|u-\uh\|_{L_3}^2\|\nabla u\|_{L_6}\,\|\nabla\phi\|_{L_6}.
\end{align*}
Next, in view of the Gagliardo--Nirenberg--Sobolev inequality,
\begin{equation}\label{Sobolev-interpolation}
\|v\|_{L_3}\le C\|v\|^{1/2}\,\|\nabla v\|^{1/2},
\end{equation}
the Sobolev Imbedding Theorem and the elliptic regularity \eqref{auxiliary-bound}, we have
\begin{equation}\label{term-II-est}
\begin{split}
| II|&\le C\|\nabla(u-\uh)\|^{3/2}\|u-\uh\|^{1/2}\|u-u_h\|+C\|\nabla(u-\uh)\|\,\|u-\uh\|\|u-\uh\|\\
&\le C\|\nabla(u-\uh)\|^{3/2}\|u-\uh\|^{1/2}\|u-u_h\|.
\end{split}
\end{equation}

Combining the previous estimates for terms $I$ and $II$, we get the 
 desired bound for $h$ sufficiently small.
\end{proof}

To complete the proof of Theorem~\ref{thm-L2-bound}, it remains to show that  the consistency error 
bound~\eqref{eq:consistency:term I} holds true. We do so through the 
following lemmas.

\begin{lemma} \label{thm:consistency:term I}
Under the assumptions of Theorem~\ref{thm-L2-bound} and given $\phi\in H^1_0(\Omega)\cap H^2(\Omega)$, there exists a positive constant $C$ independent of $h$ such that 
\begin{equation*}
|a(u;u,\phi)-a(\uh;\uh,\phi)|
\le
h\|\nabla(\u-\uh)\|
+\|\u-\uh\|^{1/2}\|\nabla(\u-\uh)\|^{3/2}
+h^{r}\|u\|_{r}+h^{r}\|f\|_{r-1}
)\|\phi\|_{2},
\end{equation*}
where $r=\min\{s,k+1\}$.
\end{lemma}
\begin{proof}
 Let $\phi_I\in \Vh$ be 
 the approximation of $\phi$ given by Theorem \ref{thm:spaceApproximation}  and using~\eqref{eq:origVariationalForm} and \eqref{vem} we split the difference  $ a(u;u,\phi)-a(\uh;\uh,\phi)$ as
\begin{align*}
 a(u;u,\phi)-a(\uh;\uh,\phi)=&\{a(u;u,\phi-\phi_I)-a(\uh;\uh,\phi-\phi_I)\}+(f-\Po{\k-1}f,\phi_I)\\
&+\{\ah(\uh;\uh,\phi_I)-a(\uh;\uh,\phi_I)\}
=I+II+III.
\end{align*}
Then, in view of~\eqref{ku}, we rewrite term $I$ as
\begin{align*}
I&=(\diff(\uh)\nabla (u-\uh)+(\diff(u)-\diff(\uh))\nabla u,\nabla(\phi-\phi_I))\\
&=(\diff(\uh)\nabla (u-\uh)+\bar \diff_u(u-\uh)\nabla u,\nabla(\phi-\phi_I)).
\end{align*}
Employing  Theorem \ref{thm:spaceApproximation} and \eqref{auxiliary-bound}, we obtain
\begin{align*}
|I|\le Ch(\|\nabla(u-\uh)\|+\|u-\uh\|\,\|\nabla u\|_{L_\infty})\|\phi\|_2\le Ch\|\nabla(u-\uh)\|\,\|\phi\|_2.
\end{align*}
As for term $II$, using Lemma \ref{lemma-vep} we get
\begin{equation}\label{eq:fterm}
|II|\le Ch^{r}\|f\|_{r-1}\|\nabla\phi_I\|\le Ch^{r}\|f\|_{r-1}\|\phi\|_2.
\end{equation}
%
%
In view of bounding term $III$, we write
\begin{align}
\nonumber
III=&
\{a_h(\uh;\uh-\u_\pi,\phi_I-\phi_\pi^1)-a(\uh;\uh-\u_\pi,\phi_I-\phi_\pi^1)\}\\
\nonumber
&+
\{a_h(\uh;\u_\pi,\phi_I-\phi_\pi^1)-a(\uh;\u_\pi,\phi_I-\phi_\pi^1)\} +
\{a_h(\uh;\uh,\phi_\pi^1)-a(\uh;\uh,\phi_\pi^1)\}\\
\label{I_32-split}
=&III_1+III_2+III_3,
\end{align}
with $\phi_\pi^1|_\E\in \PE{1}$ and $\u_\pi|_\E\in\PE{k}$, for any $\E\in\Th$  given by Theorem \ref{thm:polynomialApproximation}.
Using Theorems \ref{thm:polynomialApproximation} and \ref{thm:spaceApproximation} we bound the term $III_1$ in \eqref{I_32-split} as 
\begin{align*}
|III_1|
&\le Ch\|\nabla(\uh-\u_\pi)\|\|\phi\|_{2}
\le Ch (\|\nabla(u-\uh)\|+h^{r-1}\|u\|_{r})\|\phi\|_{2}.
\end{align*}

Next, to estimate $III_2$, we split this term as a summation over each $E\in\Th$ and use the polynomial consistency~\eqref{eq:consistency} and the definition of $\bar \diff_u$, given by~\eqref{ku}, to get
\begin{align*}
a_h^E(&\uh ;\u_\pi,\phi_I-\phi_\pi^1)-a^E(\uh;\u_\pi,\phi_I-\phi_\pi^1)\\
&=\int_E(\diff(\Po{}\uh)\nabla \u_\pi\cdot\vPo{}\nabla(\phi_I-\phi_\pi^1)-\diff(\uh)\nabla \u_\pi\cdot\nabla(\phi_I-\phi_\pi^1)\dx\\
&=\int_E(\diff(\Po{}\uh)\nabla \u_\pi\cdot(\vPo{}-\bm I)\nabla(\phi_I-\phi_\pi^1)
+(\diff(\Po{}\uh)-\diff(\uh))\nabla \u_\pi\cdot\nabla(\phi_I-\phi_\pi^1))\dx\\
&=III_2^1+III_2^2.
\end{align*}
Then, following the steps for the estimation of $I_4$ in \eqref{I4-bound}, using Theorems \ref{thm:polynomialApproximation} and \ref{thm:spaceApproximation}  along with \eqref{auxiliary-bound}, we can see that 
\begin{equation}\label{eq:III21}
| III_2^1 |\le Ch(h^{r-1}\|u\|_{r,E}+\|\Po{}\uh-\u\|_E) \|\phi\|_{2,E}.
\end{equation}

To bound $III_2^2$, we first note, in view of \eqref{Holder-ineq}, that
\begin{equation}\label{eq:B2}
| III_2^2 |
\le C \|\Po{}\uh-\uh\|_{L_3(E)}\|\nabla \u_\pi\|_{L_6(E)}\|\nabla(\phi_I- \phi_\pi^1)\|_{E}.
\end{equation}
Further, using the stability property of $P_h$, namely $\|P_h\phi_I\|_{L_3(E)}\le \tilde C\|\phi_I\|_{L_3(E)}$, and the Gagliardo--Nirenberg--Sobolev inequality \eqref{Sobolev-interpolation}, we obtain
\begin{equation}
\|\Po{}\uh-\uh\|_{L_3(E)}\le C\| \u_\pi-\uh\|_E^{1/2}\|\nabla( \u_\pi-\uh)\|_E^{1/2},
\end{equation}
with $C,\tilde{C}>0$ independent of $\E$.
Using this in~\eqref{eq:B2} and summing this new bound of \eqref{eq:B2}  and~\eqref{eq:III21} over all $\E\in\Th$ and  using Theorems \ref{thm:polynomialApproximation} and \ref{thm:spaceApproximation}, it follows that
\begin{equation*}
| III_2 |\le Ch(\|\nabla( u-\uh)\|+\|\Po{}\uh-\u\|+h^{r-1}\|u\|_{r})\|\phi\|_{2}.
\end{equation*}
 
Finally, as a consequence of Lemma~\ref{thm:consistency:term B} below, we have
\begin{equation*}
|III_3|\le C(
\|\u-\uh\|^{1/2}\|\nabla(\u-\uh)\|^{3/2}
+h^{r}\|u\|_{r}
)\|\phi\|_{2}.
\end{equation*}
Combining this with~\eqref{eq:fterm}, the bounds for $III_1$, and $III_2$, the 
 desired bound follows.
\end{proof}

\begin{lemma}\label{thm:consistency:term B}
Let the assumptions of Theorem~\ref{thm-L2-bound} hold true and $\phi\in H^2\cap H^1_0$.
Then,  there exists a constant $C$ independent of $h$ such that,
 \begin{equation*}
|a_h(\uh;\uh,
\phi_\pi^1)-a(\uh;\uh,\phi_\pi^1)|\le C(
\|\nabla(\u-\uh)\|
+\|\u-\uh\|^{1/2}\|\nabla(\u-\uh)\|^{3/2}
+h^{r}\|u\|_{r}
)\|\phi\|_{2},
\end{equation*}
where $\phi_\pi^1\in \mathbb{P}_1(E)$  for all $\E\in\Th$, is given by Theorem \ref{thm:polynomialApproximation},
  and  $r=\min\{s,k+1\}$.
\end{lemma}
\begin{proof}
Using polynomial consistency \eqref{eq:consistency}, the fact that $\vPo{}\nabla \u_\pi=\nabla \u_\pi$, with $\u_\pi\in\PE{k}$ given by Theorem~\ref{thm:polynomialApproximation} and the definition of $\bar \diff_u$ given by~\eqref{ku}, we have for all $E\in\Th$
\begin{align*}
a_h^E&(\uh;\uh, \phi_\pi^1)-a^E(\uh;\uh, \phi_\pi^1)
=\int_E\diff(\Po{}\uh)(\vPo{}\nabla\uh)\cdot\nabla \phi_\pi^1-\diff(\uh)\nabla\uh\cdot\nabla \phi_\pi^1\dx\\
&=\int_E\diff(\Po{}\uh)(\vPo{}-\bm I)\nabla\uh\cdot\nabla \phi_\pi^1
+(\diff(\Po{}\uh)-\diff(\uh))\nabla\uh\cdot\nabla \phi_\pi^1\dx\\
&=\int_E\diff(\Po{}\uh)(\vPo{}-\bm I)\nabla(\uh-\u_\pi)\cdot\nabla \phi_\pi^1\dx
+\int_E\bar \diff_u(\Po{}\uh-\uh)\nabla\uh\cdot\nabla \phi_\pi^1\dx\\
&=\int_E(\diff(\Po{}\uh)-\diff(u))(\vPo{}-\bm I)\nabla(\uh-\u_\pi)\cdot\nabla \phi_\pi^1\dx
+\int_E\diff(u)(\vPo{}-\bm I)\nabla(\uh-\u_\pi)\cdot\nabla \phi_\pi^1\dx\\
&+\int_E\bar \diff_u(\Po{}\uh-\uh)\nabla\uh\cdot\nabla \phi_\pi^1\dx\\
&=\int_E\bar \diff_u(\Po{}\uh-u)(\vPo{}-\bm I)\nabla(\uh-\u_\pi)\cdot\nabla \phi_\pi^1\dx
+\int_E\diff(u)(\vPo{}-\bm I)\nabla(\uh-\u_\pi)\cdot\nabla \phi_\pi^1\dx\\
&+\int_E\bar \diff_u(\Po{}\uh-\uh)\nabla\uh\cdot\nabla \phi_\pi^1\dx
=I_{\E}+II_{\E}+III_{\E}.
\end{align*}
Let $I=\sum_{\E}I_{\E}$, then 
 we easily get
\begin{align*}
| I |\le C\|\Po{}\uh-u\|_{L_3}\|\nabla \phi_\pi^1\|_{L_6}\|\nabla(\uh- \u_\pi)\|.
\end{align*}
Using Theorem \ref{thm:polynomialApproximation}, we have $\|\nabla \phi_\pi^1\|_{L_6}\le C \|\nabla \phi\|_{W^{1,6}}$ and, hence, 
using a Sobolev imbedding, 
\begin{equation}\label{Sobolev-embed}
\|\nabla \phi_\pi^1\|_{L_6}\le C|\phi|_{2}.
\end{equation}
Now, using Theorem \ref{thm:polynomialApproximation} once again, we get
\begin{align*}
| I |\le C(\| \u_\pi-\uh\|^{1/2}\|\nabla( \u_\pi-\uh)\|^{3/2} +h^{r-1/2}\|\nabla( \u_\pi-\uh)\|)\|\phi\|_{2}.
\end{align*}
To bound $ II_{\E}$, we rewrite this term as
\begin{align*}
II &=\int_E\diff(u)(\vPo{}-\bm I)\nabla(\uh-\u_\pi)\cdot\nabla( \phi_\pi^1-\phi)\dx
+\int_E\diff(u)(\vPo{}-\bm I)\nabla (\uh-\u_\pi)\cdot\nabla\phi\dx\\
&=\int_E\diff(u)(\vPo{}-\bm I)\nabla(\uh-\u_\pi)\cdot\nabla( \phi_\pi^1-\phi)\dx
+\int_E(\vPo{}-\bm I)(\diff(u)\nabla\phi)\nabla (\uh-\u_\pi)\dx
\end{align*}
Then for $II=\sum_{\E}II_E$, using Theorem \ref{thm:polynomialApproximation}, it immediately follows that
\begin{align*}
| II |\le Ch\|\nabla(\uh-\u_\pi)\|\|\phi\|_{2}.
\end{align*}

Next, we consider the term $III_{\E}$, which can be rewritten as
\begin{align*}
III_{\E}
&=\int_E(\Po{}\uh-\uh)\bar \diff_u[\nabla(\uh-\u_\pi)\cdot\nabla \phi_\pi^1
+\nabla \u_\pi\cdot\nabla \phi_\pi^1]\dx=III_{\E,1}+III_{\E,2}.
\end{align*}
Then using  H\"older inequality \eqref{Holder-ineq}, and we obtain for $III_1=\sum_{\E}III_{\E,1}$
\begin{equation*}
| III_1 |
\le C\|\Po{}\uh-\uh\|_{L_3}\|\nabla \phi_\pi^1\|_{L_6}\|\nabla(\uh- \u_\pi)\|.
\end{equation*}
Hence, following the steps  in the proof of Lemma~\ref{thm:consistency:term I}, and using~\eqref{Sobolev-embed},  we get
\begin{align*}
| III_1 |\le C\| \u_\pi-\uh\|^{1/2}\|\nabla( \u_\pi-\uh)\|^{3/2} \|\phi\|_{2}.
\end{align*}

Next, in view of the fact that $\nabla \u_\pi\cdot\nabla \phi_\pi^1\in \PE{k}$,  we have
\begin{equation}
III_{\E,2}=\int_E(\Po{}\uh-\uh)(\bar\diff_u-c)\nabla \u_\pi\cdot\nabla \phi_\pi^1,\quad\forall c\in\mathbb R.
\end{equation}
Thus, for $III_2=\sum_{\E}III_{\E,2}$, we get
\begin{align*}
| III_2 |&\le Ch\|\uh-\Po{}\uh\|_{L_3}\|\nabla \phi_\pi^1\|_{L_6}\|\nabla \u_\pi\|.
\end{align*}
 Therefore,  Theorem \ref{thm:polynomialApproximation}, and the Sobolev inequalities \eqref{Sobolev-interpolation}, \eqref{Sobolev-embed}, give
\begin{align*}
| III_2 |&\le C
           h\|\uh-\u_\pi\|^{1/2}\|\nabla(\uh-\u_\pi)\|^{1/2}\|\phi\|_{2}.
\end{align*}
Collecting the above bounds, yields for $III=III_{1}+III_{2}$
\begin{align*}
| III |&\le C(
h\|\uh-\u_\pi\|^{1/2}\|\nabla(\uh-\u_\pi)\|^{1/2}
+\|\uh-\u_\pi\|^{1/2}\|\nabla(\uh-\u_\pi)\|^{3/2})\|\phi\|_{2}.
\end{align*}
Therefore
\begin{align*}
|a_h(\uh;\uh,\phi_\pi^1)-a(\uh;\uh,\phi_\pi^1)|&\le 
C(
 h\|\nabla(\uh-\u_\pi)\|
+\|\uh-\u_\pi\|^{1/2}\|\nabla(\uh-\u_\pi)\|^{3/2})\|\phi\|_{2},
\end{align*}
from which the desired bound follows  using once again Theorem~\ref{thm:polynomialApproximation}. 
\end{proof}

Having concluded the proof of Theorem \ref{thm-L2-bound}, in order to show optimal convergence rate of the error in $H^1$ and $L^2$-norms,  it remains to demonstrate that  $\uh$ converge to $u$.

\begin{theorem}\label{thm-converge} 
Under the same assumptions as in Theorems~\ref{thm-H1-bound} and \ref{thm-L2-bound}, the VEM solution $\uh$ converges to the exact solution $u$ in $H^{1}_{0}(\Omega)$.
\end{theorem}
\begin{proof} From Theorem~\ref{thm-existence} it follows that $\|\nabla\uh\|$ is bounded from above.
Therefore, we can choose a subsequence $u_{h_{k}}$ such that for some $z\in H^1_0(\Omega)$, $u_{h_{k}}\to z$, weakly  in $H^{1}_{0}(\Omega)$, as $h_k\to 0$ and, thus, strongly in $L^2(\Omega)$.
Also, for arbitrary $v\in C^{\infty}_{0}(\Omega)$ let $v_{h_k}$ be a sequence in $V_{h_k}$ such that 
\begin{equation}
\|\nabla(v-v_{h_k})\|\to0,\quad h_k\to0.
\end{equation}
Then
\begin{align*}
&|a(z;z,v)-(f,v)|\le |(\diff(z)\nabla z,\nabla(v-v_{h_k})|\\
&\qquad +|(\diff(z)\nabla z,\nabla v_{h_k})-a_h(u_{h_k};u_{h_k},v_{h_k})|+|(\Po{\k-1}f,v_{h_k}-v)| +|\vep(f,v)|\\
&\qquad \le C\|\nabla (v-v_{h_k})\|+|(\diff(z)\nabla z,\nabla v_{h_k})-a_h(u_{h_k};u_{h_k},v_{h_k})|+Ch_k\|f\|_1\|v\|.
\end{align*}
Thus, if 
\begin{equation}\label{w-weak-sol}
|(\diff(z)\nabla z,\nabla v_{h_k})-a_h(u_{h_k};u_{h_k},v_{h_k})|\to0,\quad h_k\to0,
\end{equation}
then $z$ is the weak solution of \eqref{eq:pde}. To show \eqref{w-weak-sol}, we rewrite its left-hand side as
\begin{align*}
&|(\diff(z)\nabla z,\nabla v_{h_k})-a_h(u_{h_k};u_{h_k},v_{h_k})|\\
&\le |(\diff(z)\nabla z-\diff(u_{h_k})\nabla u_{h_k},\nabla v_{h_k})|
+|(\diff(u_{h_k})\nabla u_{h_k},\nabla v_{h_k})-a_h(u_{h_k};u_{h_k},v_{h_k})|\\
&\le C\|\nabla (v-v_{h_k})\|+|(\diff(z)\nabla (z-u_{h_k}),\nabla v)|+|((\diff(z)-\diff(u_{h_k}))\nabla u_{h_k},\nabla v)|\\
&
+|(\diff(u_{h_k})\nabla u_{h_k} ,\nabla v_{h_k})-a_h(u_{h_k};u_{h_k},v_{h_k})|
\end{align*}
Using the fact that $u_{h_k}\to z$, and $v_{h_k}\to v$, we see that \eqref{w-weak-sol} holds.
 Hence $a(z;z,v)=(f,v)$, and thus $u=z$, since $u$ is the unique solution of \eqref{eq:pde}. Then, it follows that 
 $\uh\to u$ in $L^2(\Omega)$. Hence, $\|u-\uh\|\to 0$ and the result follows from Theorems~\ref{thm-L2-bound}, and~\ref{thm-H1-bound}.
 \end{proof}

In view of Theorems \ref{thm-H1-bound}, \ref{thm-L2-bound} and \ref{thm-converge}, the following a priori error estimates now readily follows.

\begin{theorem}
Let $u\in H^1_0(\Omega)$ be the solution of \eqref{eq:pde} and suppose that $u\in H^{s}(\Omega)\cap W^1_\infty (\Omega)$, $s\ge2$, assuming that $f\in H^{s-1}(\Omega)$ and $\diff(u)\in W^{s-1}_{\infty}(\Omega)$.
Let also $\uh\in \Vh$ be the solution of \eqref{vem}.
Then, there exists a constant $C$ independent of $h$ such that, for $h$ sufficiently small, 
 \begin{equation}
 \|u-\uh\|+h\|\nabla(u-\uh)\|\le Ch^{r},
\end{equation}
where $r=\min\{k+1,s\}$.
\end{theorem}

\section{Iteration method}
\label{sec:iteration}

 In this section we show that, given a virtual element space $\Vh$, the sequence of solutions we obtain using fixed point iterations to solve the VEM problem~\eqref{vem} converges to the true solution $\uh\in\Vh$ of \eqref{vem}.
 
Starting with a given $\uh^0\in\Vh$ we construct a sequence $\uh^n$, $n\ge0$,  such that 
\begin{equation}\label{vem-iteration}
a_h(\uh^n;\uh^{n+1},\vh)=(\Po{\k-1} f,\vh), \quad\forall \vh\in V_h. 
\end{equation}

The convergence in $H^1$ of the sequence  $\uh^n$ as $n\to\infty$ to a fixed point of~\eqref{vem-iteration}, and hence a solution of~\eqref{vem},  is an immediate consequence of the following result.
\begin{theorem}\label{thm:convergent-seq}
Let $\{\uh^n\}\subset \Vh$ be the sequence produced in \eqref{vem-iteration}, then 
\begin{equation}
\|\nabla(\uh^n-\uh^{n+1})\|\to0,\quad \text{as }\  n\to \infty.
\end{equation}
\end{theorem}
\begin{proof}
In view of Assumption \ref{ass:bilinearForms} and the fact that $\ah(\uh^n;\cdot,\cdot)$ is symmetric, we have
\begin{equation}\begin{split}
c_\star\|\nabla(\uh^n-\uh^{n+1})\|^2&\le \ah(\uh^n;\uh^n-\uh^{n+1},\uh^n-\uh^{n+1})\\
&=\ah(\uh^n;\uh^n,\uh^n)-2\ah(\uh^n;\uh^{n+1},\uh^n)+\ah(\uh^n;\uh^{n+1},\uh^{n+1}),
\end{split}\end{equation}
with $c_\star=\elipLower\diffStabLower$.
Then using \eqref{vem-iteration}, we obtain 
\begin{equation*}
\ah(\uh^n;\uh^{n+1},\uh^n)=(\Po{\k-1} f,\uh^n-\uh^{n+1})+\ah(\uh^n;\uh^{n+1},\uh^{n+1}),
\end{equation*}
giving
\begin{equation}\begin{split}
c_\star\|\nabla(\uh^n-\uh^{n+1})\|^2&\le 
\ah(\uh^n;\uh^n,\uh^n)-2(\Po{\k-1} f,\uh^n-\uh^{n+1})-\ah(\uh^n;\uh^{n+1},\uh^{n+1})\\
&=\mathcal{F}(\uh^n)-\mathcal{F}(\uh^{n+1}),
\end{split}\end{equation}
where 
$\mathcal{F}(v)=\ah(\uh^n;v,v)-2(\Po{\k-1} f,v)$.
Therefore, $\mathcal{F}(\uh^n)$ is a decreasing sequence and, in view of the fact that 
\begin{equation}
\mathcal{F}(v)=\ah(\uh^n;v,v)-2(\Po{\k-1} f,v)\ge \elipLower\|\nabla v\|^2-2\| f\|\|\nabla v\|\ge -\dfrac{\|f\|^2}{\elipLower},
\end{equation}
$\mathcal{F}(\uh^n)$ is bounded from below. Therefore $\mathcal{F}(\uh^n)-\mathcal{F}(\uh^{n+1})\to0$, as $n\to \infty$, which completes the proof.
\end{proof}

\section{Numerical results}\label{sec:num}

In order to test the VEM proposed in Section~\ref{sec:vem} we need to specify a bilinear form  satisfying Assumption~\ref{ass:bilinearForms}.  We fix $a_h^\E$ as follows:
\begin{equation*}
	\ahE(\zh;\vh, \wh) = \intE \diff(\Po{}\zh) (\vPo{} \nabla \vh)\cdot (\vPo{} \nabla \uh) \dx+
	\SaE(\zh;(I-\Po{})\vh, (I-\Po{})\wh),
\end{equation*}
with  the VEM stabilising form $S^E$ given by

\begin{equation*}
\SaE(\zh;(I-\Po{})\vh, (I-\Po{})\wh) := {\diff_\E(\Po{0,\E}\zh)} h_{\E}^{\spacedim - 2} 
\overrightarrow{(I-\Po{})\vh}\cdot\overrightarrow{(I-\Po{})\wh}.
\end{equation*}
here, $I$ denotes the identity operator, $\overrightarrow\vh$ is the vector with entries the degrees of freedom of $\vh\in\VhE$, and $\overrightarrow\vh\cdot\overrightarrow\wh$ is the euclidean scalar product of the degrees of freedom of $\vh, \wh\in\VhE$.

The above definition of the local bilinear form extends to the nonlinear setting the one considered in~\cite{Confnonconf} and, similarly to the linear case, it is straightforward to show that it satisfies the stability condition~\eqref{stability}.
Following~\cite{BasicPaper} instead, the projector $\PN{\ell}$ can be used in place of $\Po{}$ in the stabilising term. The practical implementation of these projector operators and VEM assembly are discussed in~\cite{Hitchhikers,Confnonconf}.

In the examples below, approximation errors are measured by comparing the piecewise polynomial quantities $\Po{k}\uh$ and $\Po{k-1}\nabla\uh$ with the exact solution $\u$ and solution's gradient $\nabla\u$, respectively.

The tests are performed using the VEM implementation within the Distributed and Unified Numerics Environment (DUNE) library~\cite{dune}, presented in~\cite{vem-dune}. 

A representative example is shown in Figure~\ref{fig:meshes}. The polygonal mesh was generated using~\cite{polymesher}. 
\begin{figure}[h!]
\hspace{-4 mm}
\includegraphics[scale=0.4]{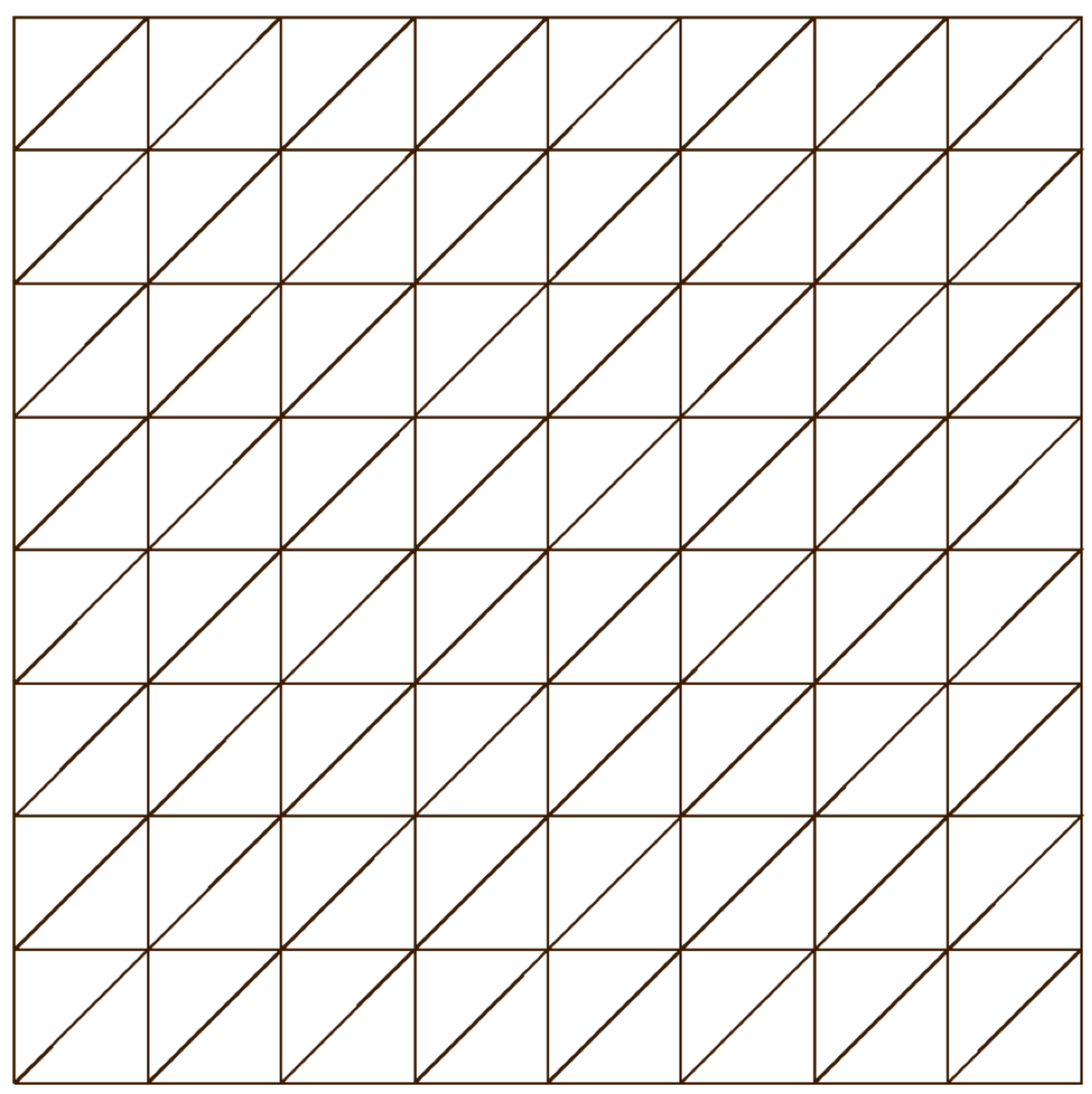}\hspace{5 mm}
\includegraphics[scale=0.4]{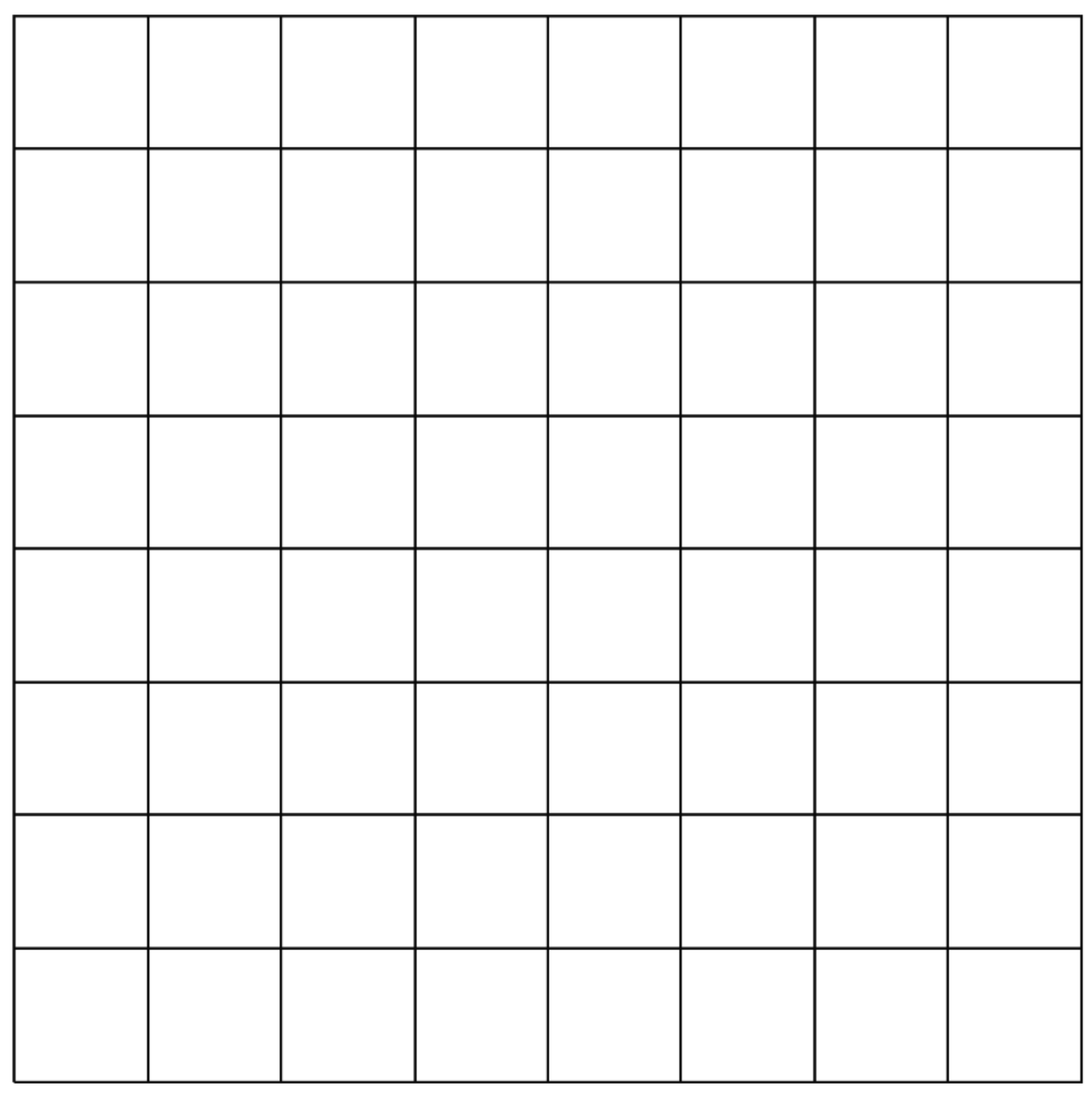}\\
\hspace{-4mm}
\includegraphics[scale=0.405]{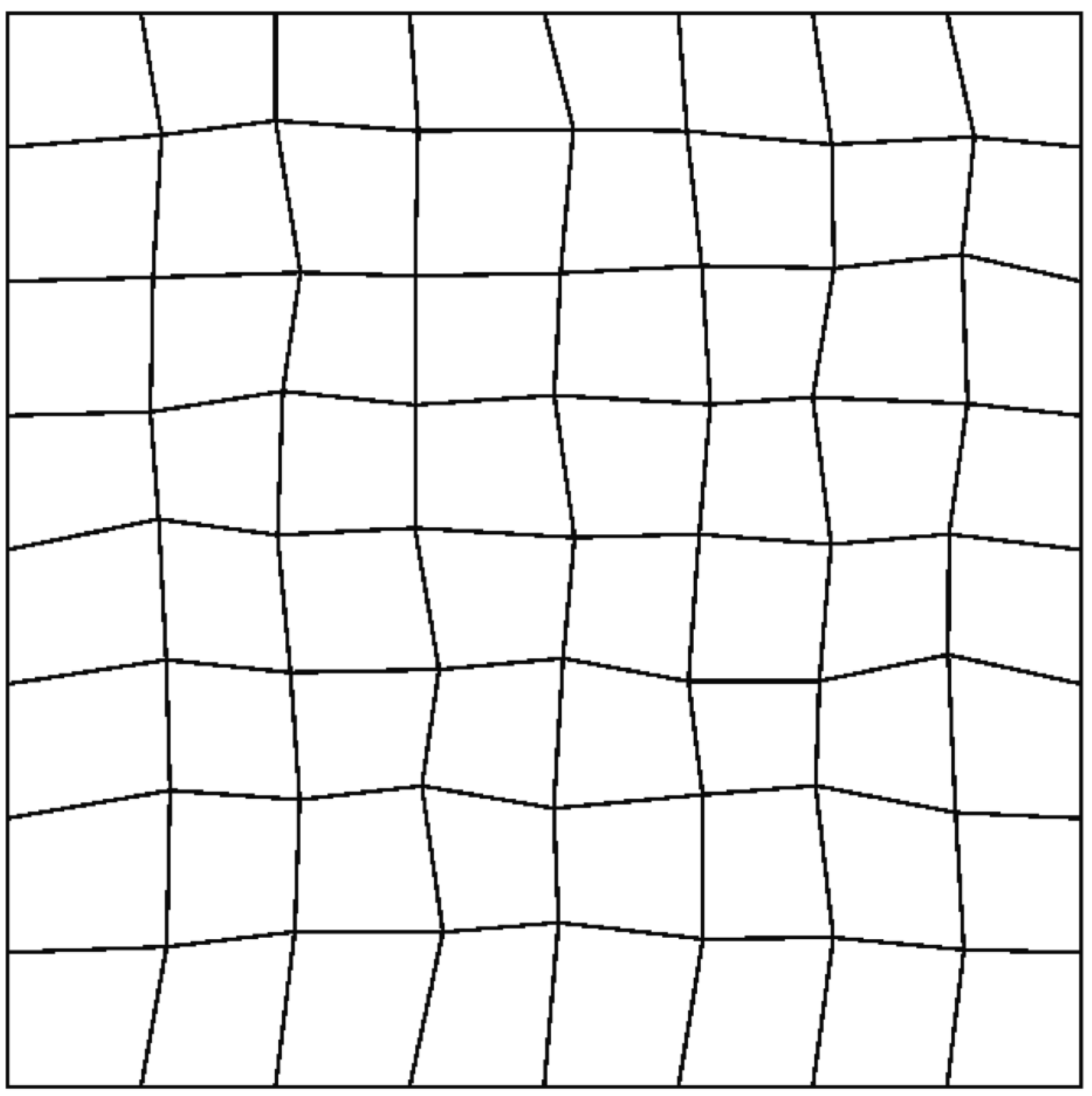}\hspace{5 mm}
\includegraphics[scale=0.65]{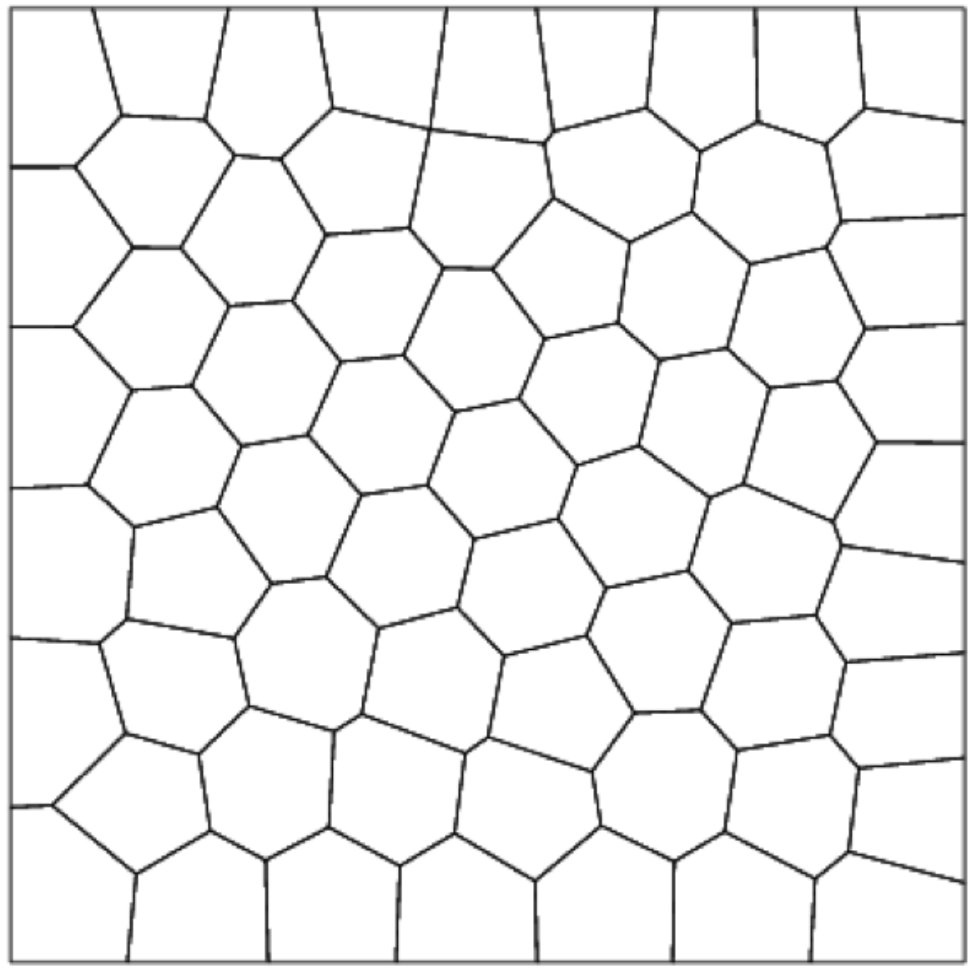}
\caption{Sample meshes used in the numerical test corresponding to an $8\times 8$ subdivision of the domain: triangles, squares, radom quads, and polygons.}
\label{fig:meshes}
\end{figure}
We use fixed point iterations analysed in Section~\ref{sec:iteration} to solve the nonlinear system resulting from the VEM discretisation. This is  compared below with Newton-Raphson iterations, defined as follows. Given an initial iterate $\uh^0\in\Vh$, we construct a sequence $\uh^{n+1}=\uh^n+\delta^n$, $n\ge0$,  by solving at each iteration the linearised problem: find $\delta^n\in\Vh$ such that 
\begin{equation}\label{vem-newton}
a_h(\uh^n;\delta^{n},\vh)+b_h(\uh^n;\delta^{n},\vh)=(\Po{\k-1} f,\vh)-a_h(\uh^n;\uh^n,\vh), \quad\forall \vh\in V_h.
\end{equation}
Here, the extra terms stemming from the linearisation of both the consistency and stability terms in $a_h$ are collected in the global form $b_h:=\sum_{\E\in\Th}b_h^\E$, with the local form $b_h^\E$, $\E\in\Th$, given by 
\begin{align*}
b_h^\E(\uh^n;\delta^{n},\vh)=&
\intE \diff_u(\Po{}\uh^n)\Po{}\delta^k (\vPo{} \nabla \uh^n)\cdot (\vPo{} \nabla \vh) \dx\\
&+h_{\E}^{\spacedim - 2} {\diff_u(\Po{0,\E}\uh^n)} \Po{0,\E} \delta^k\, \overrightarrow{{u_h^n}-\Po{}u_h^n}\cdot \overrightarrow{\vh-\Po{}\vh}.
\end{align*}

{\bf Numerical test.} We consider the following test problem from~\cite{FVEM}. We solve~\eqref{eq:pde} on $\Omega=[0,1]^2$ with $\diff (\u)=1/(1+\u)^2$ and the function $f$ chosen such that the exact solution is $\u=(x-x^2)(y-y^2)$. Note that, although the diffusion coefficient is not even bounded on the whole of $\Re$, it is smooth in a neighbourhood of the range of $\u$. 
As initial guess for the nonlinear solve we use the constant zero function  and the conjugate-gradient method is used to solve the linear system at each iteration. 
The relative errors for the approximation of $\u$ and its
gradient as a function of the mesh size $h$ are shown in Table~\ref{tab:1_poly} for $k=1$ and a sequence of polygonal meshes, cf. the right-most plot in Figure~\ref{fig:meshes}. The numerical results confirm the theoretical rate of convergence. The table also displays the number of fixed point and Newton-Raphson iterations performed until the indicated stopping criteria is reached. 

\begin{table}[h!]\label{tab:1_poly}
\begin{center}
\begin{tabular}{|c|c|c|c|c|c|c|}
\hline
DOF & $\|u-\Po{k}\uh\|$ & EOC & $\|\nabla u-\Po{k-1}\nabla\uh\|$   & EOC & FP & NR \\ \hline
9 & 1.30E-02 & --    & 9.44E-02            & -- & 6 & 4 \\ \hline
34 & 3.40E-03 & 2.018 & 4.96E-02      & 0.967 & 7 & 4 \\ \hline
129 & 8.16E-04 & 2.140 & 2.51E-02    & 1.022 & 6 & 4 \\ \hline
510 & 1.89E-04 & 2.131 & 1.25E-02    & 1.012 & 6 & 4 \\ \hline
2042 & 4.49E-05 & 2.070 & 6.26E-03  & 1.001 & 6 & 3 \\ \hline
8162 & 1.11E-05 & 2.011 & 3.12E-03  & 1.006 & 6 & 3 \\ \hline
\end{tabular}
	\end{center}
\caption{Errors and empirical order of convergence (EOC) on a sequence of polygonal meshes. The  Fixed Point and Newton-Raphson  iterations needed to reach the tolerance $10^{-10}$ are reported in the right-most columns.}
\end{table}

The convergence history with respect to all meshes in Figure~\ref{fig:meshes} are reported in the loglog plots of Figure~\ref{fig:errors} showing that the performance is similar in all cases. Note that, as $k=1$, in the case of the sequence of triangular meshes, the VEM coincides with the standard linear finite element method.
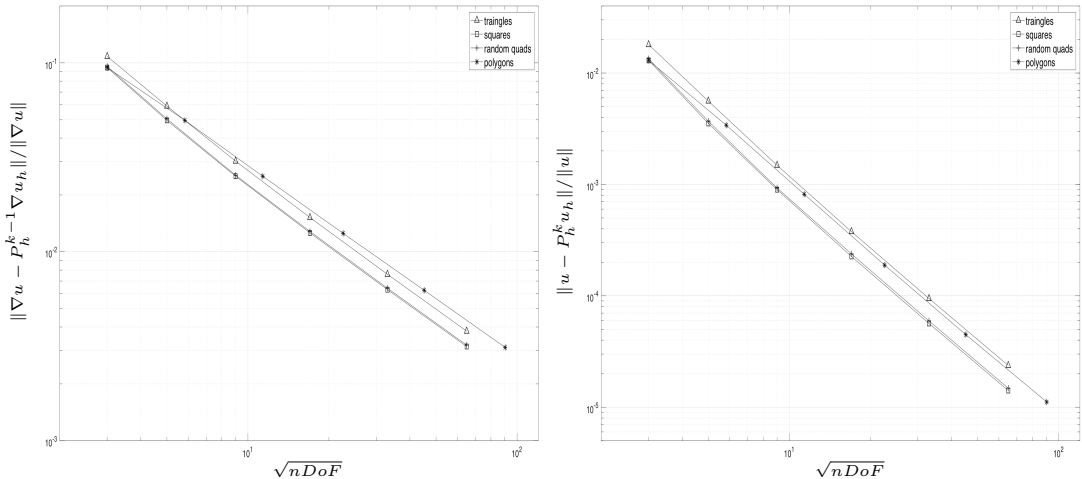
\begin{figure}[h!]
\hspace{-2mm}
\includegraphics[width=6.5cm,height=6cm]{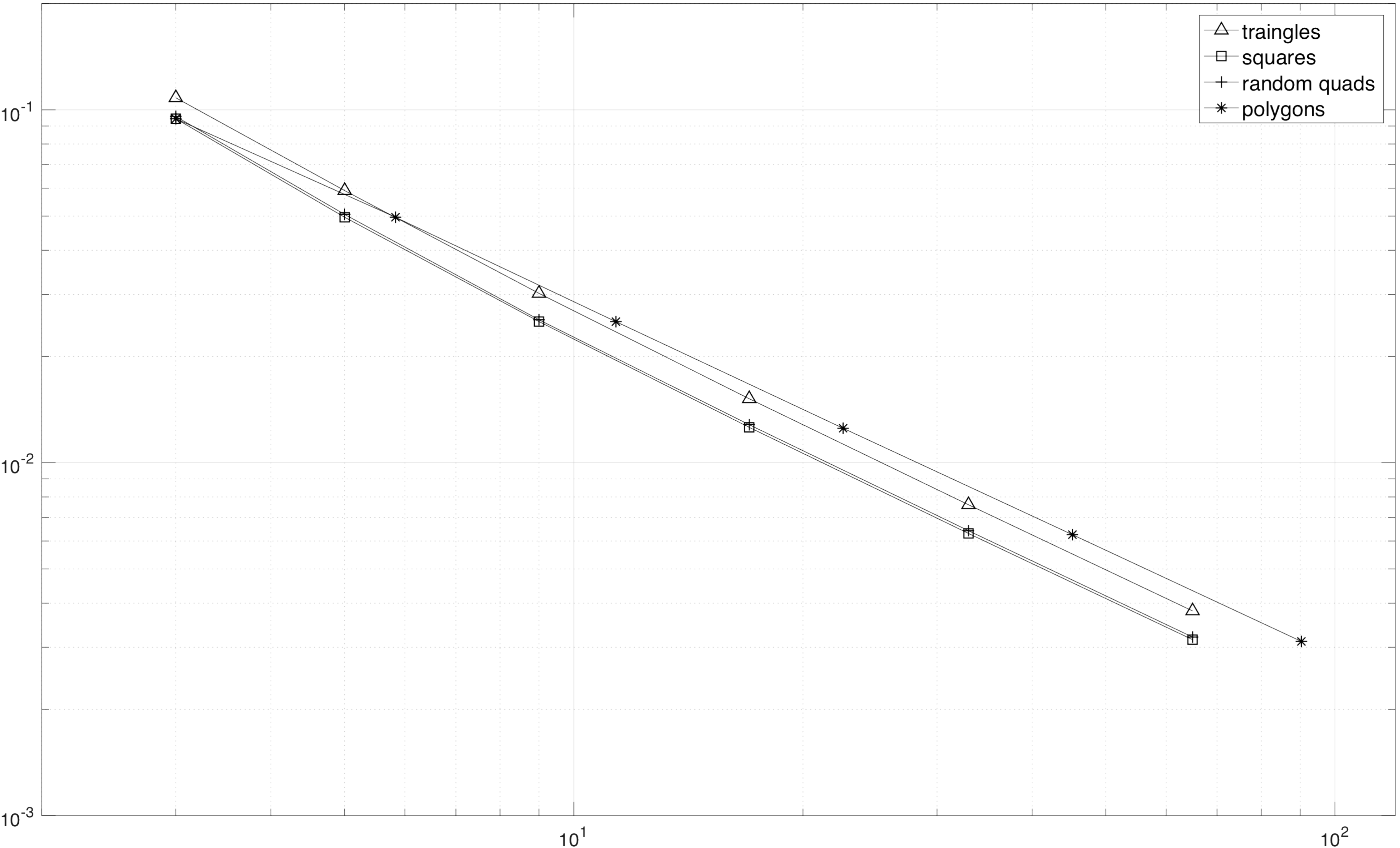}\hspace{5mm}
\includegraphics[width=6.5cm,height=6cm]{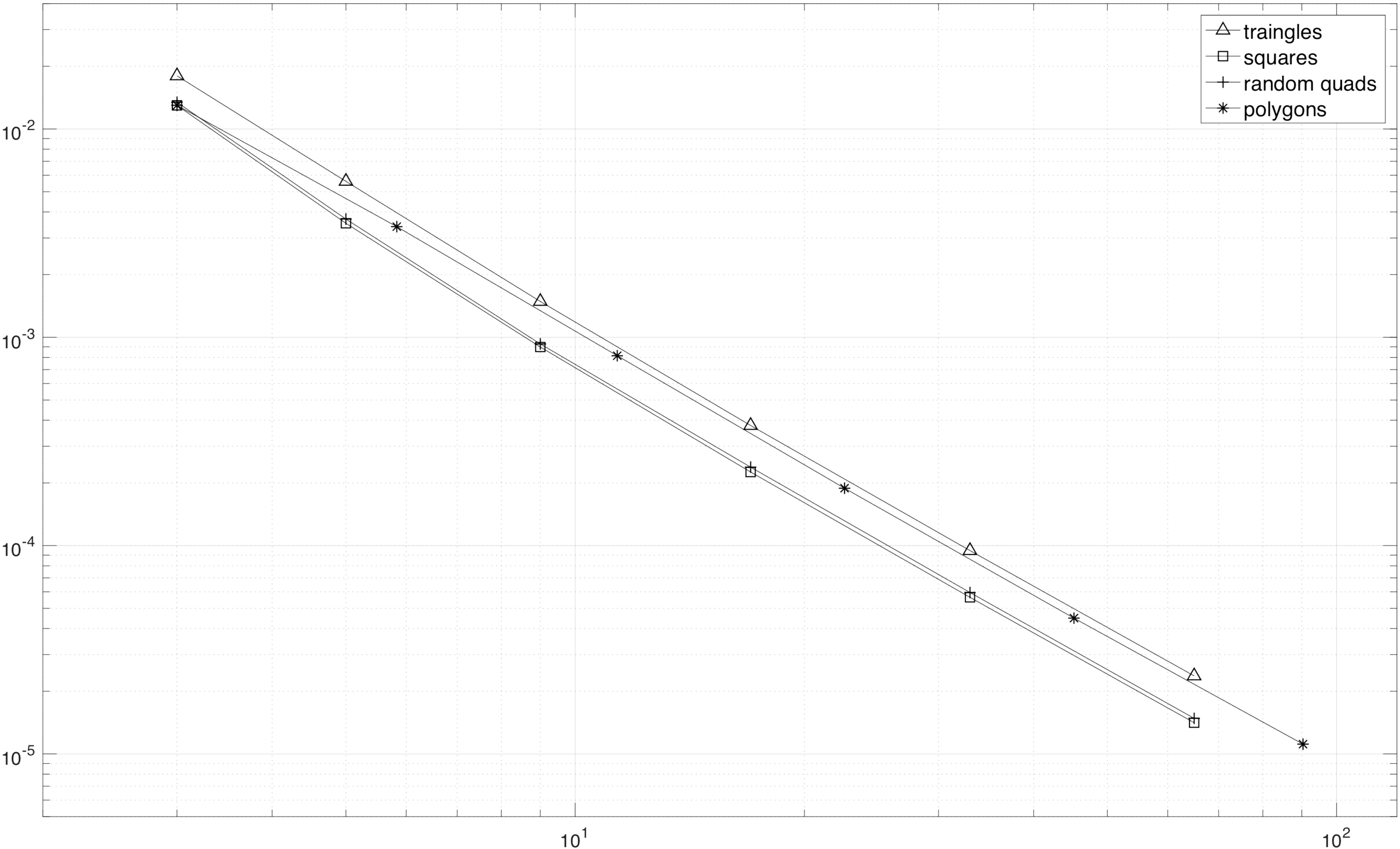}
\caption{Convergence history for $\k=1$ and the sequences of meshes represented in Figure~\ref{fig:meshes}.}
\label{fig:errors}
\begin{picture}(0,0)
\put(-210,95){\rotatebox{90}{{\tiny $\|\nabla u-\Po{k-1}\nabla\uh\|/\|\nabla u\|$ }}}
\put(-5,105){\rotatebox{90}{{\tiny $\|u-\Po{k}\uh\|/\| u\|$ }}}
\put(-110,37){{\tiny $\sqrt{nDoF}$}}
\put(93,37){{\tiny $\sqrt{nDoF}$}}
\end{picture}
\end{figure} 

\section{{Conlusions}}
\label{sec:conc}

With this paper, we show that the Virtual Element Method can be extended  to nonlinear problems. In particular, we consider elliptic quasilinear problems with Lipschitz continuous diffusion in two and three dimensions and show that it suffices to evaluate the diffusion coefficient with the component of the VEM solution which is readily accessible. 
We prove optimal order a priori error estimates under the same mesh assumptions used in the linear setting.

\section*{Acknowledgements}

This research was initiated during the visit of PC to Leicester funded by the LMS Scheme 2 grant (Project RP201G0158). AC was partially supported by the EPSRC (Grant EP/L022745/1).
EHG was supported by a Research Project Grant from The Leverhulme Trust (grant no. RPG 2015-306).
All this support is gratefully acknowledged. We also express our gratitude to Martin Nolte (Albert-Ludwigs-Universit\"at Freiburg) and Andreas Dedner (University of Warwick) for supporting the implementation of the VEM within DUNE-FEM.

\end{document}